\documentclass{article}

\usepackage[latin1]{inputenc}
\usepackage{t1enc}
\usepackage{amsmath, amssymb, amsfonts, amsthm, amsopn,epsfig}
\usepackage[none]{hyphenat}
\usepackage{tikz}
\usepackage{float}
\usepackage{graphicx}
\usepackage{caption}
\usepackage[margin=1in]{geometry}
\usepackage[toc,page]{appendix}
\usepackage{epstopdf}
\usepackage{etoolbox}
\usepackage[colorlinks=true]{hyperref}

\theoremstyle{plain}
\newtheorem{theorem}{Theorem}
\newtheorem{definition}[theorem]{Definition}

\newtheorem{claim}[theorem]{Claim}
\newtheorem{lemma}[theorem]{Lemma}
\newtheorem{conjecture}[theorem]{Conjecture}

\theoremstyle{definition}

\DeclareMathOperator*{\F}{\mathcal{F}}
\DeclareMathOperator*{\M}{\mathcal{M}}
\DeclareMathOperator*{\C}{\mathcal{C}}

\author{
Andrey Kupavskii \thanks{\'{E}cole Polytechnique F\'{e}d\'{e}rale de Lausanne, Research partially supported by Swiss National Science Foundation grants no. 200020-162884 and 200021-175977. \emph{e-mail}: \textbf{\{andrei.kupavskii, janos.pach, istvan.tomon\}@epfl.ch}} \thanks{Moscow Institute of Physics and Technology, Research partially supported by the grant N 15-01-03530 of the Russian Foundation for Basic Research.}
\and
J\'{a}nos Pach \footnotemark[1]
\and
Istv\'{a}n Tomon \footnotemark[1]	
}

\title{On the size of $k$-cross-free families}

\begin{document}
\sloppy

\maketitle

\begin{abstract}
	Two subsets $A,B$ of an $n$-element ground set $X$ are said to be \emph{crossing}, if none of the four sets $A\cap B$, $A\setminus B$, $B\setminus A$ and $X\setminus(A\cup B)$ are empty. It was conjectured by Karzanov and Lomonosov forty years ago that if a family $\F$ of subsets of $X$ does not contain $k$ pairwise crossing elements, then $|\F|=O_{k}(n)$. For $k=2$ and $3$, the conjecture is true, but for larger values of $k$ the best known upper bound, due to Lomonosov, is $|\F|=O_{k}(n\log n)$. In this paper, we improve this bound by showing that $|\F|=O_{k}(n\log^{*} n)$ holds, where $\log^{*}$ denotes the iterated logarithm function.
\end{abstract}

\section{Introduction}
 As usual, denote $[n]:=\{1,\ldots,n\}$ and let $2^{[n]}$ be the family of all subsets of $[n]$. Two sets $A,B\in 2^{[n]}$ are said to be \emph{crossing}, if $A\setminus B$, $B\setminus A$, $A\cap B$ and $[n]\setminus(A\cup B)$ are all non-empty.

 We say that a family $\F\subset 2^{[n]}$ is \emph{$k$-cross-free} if it does not contain $k$ pairwise crossing sets. The following conjecture was made by Karzanov and Lomonosov~\cite{KL}, \cite{K} and later by Pevzner~\cite{P}; see also Conjecture 3 in ~\cite{BMP}, Section 9.
 
 \begin{conjecture}\label{mainconj}
 	Let $k\geq 2$ and $n$ be positive integers, and let $\F\subset 2^{[n]}$ be a $k$-cross-free family. Then $|\F|=O_{k}(n)$.
 \end{conjecture}

Here and in the rest of this paper, $f(n)=O_{k}(n)$ means that $f(n)\leq c_kn$ for a suitable constant $c_k>0$, which may depend on the parameter $k$. 
\smallskip

It was shown by Edmonds and Giles~\cite{EG} that every $2$-cross-free family $\F\subset 2^{[n]}$ has at most $4n-2$ members. Pevzner~\cite{P} proved that every $3$-cross-free family on an $n$-element underlying set has at most $6n$ elements, and Fleiner~\cite{F} established the weaker bound $10n$, using a simpler argument.  For $k>3$, Conjecture~\ref{mainconj} remains open. The best known general upper bound for the size of a $k$-cross-free family is $O_{k}(n\log n)$, which can be obtained by the following elegant argument, due to Lomonosov. 

Let $\F\subset 2^{[n]}$ be a maximal $k$-cross-free family. Notice that for any set $A\in\F$, the complement of $A$ also belongs to $\F$. Thus, the subfamily
$${\F}'=\{ A\in\F : |A|<n/2\}\cup\{A\in\F : |A|=n/2\;{\rm and}\; 1\in A\}$$
contains precisely half of the members of $\F$. For every $s,\, 1\le s\le n/2$, any two $s$-element members of $\F'$ that have a point in common, are crossing. Since $\F'$ has no $k$ pairwise crossing members, every element of $[n]$ is contained in at most $k-1$ members of $\F'$ of size $s$. Thus, the number of $s$-element members is at most $(k-1)n/s$, and 
$$|{\F}'|={|\F|}/{2}\le 1+\sum_{s=1}^{n/2}(k-1)n/s=O_{k}(n\log n).$$

 The main result of the present note represents the first improvement on this 40 years old bound. Let $\log_{(i)} n$ denote the function $\log\ldots\log n$, where the $\log$ is iterated $i$ times, and let $\log^{*} n$ denote the \emph{iterated logarithm of $n$}, that is, the largest positive integer $i$ such that $\log_{(i)}n>1$.

\begin{theorem}\label{mainthm}
	Let $k\geq 2$ and $n$ be positive integers, and let $\F\subset 2^{[n]}$ be a $k$-cross-free family. Then $|\F|=O_{k}(n\log^{*}n).$
\end{theorem}

Conjecture~\ref{mainconj} has been proved in the following special case. Let $\F$ be a $k$-cross-free family consisting of contiguous subintervals of the cyclic sequence $1, 2,\ldots, n$. It was shown by Capoyleas and Pach \cite{CaP} that in this case
$$|\F|\leq 4(k-1)n-2\binom{2k-1}{2},$$
provided that $n\geq 2k-1$. This bound cannot be improved.
\smallskip

A {\em geometric graph} $G$ is a graph drawn in the plane so that its vertices are represented by points in general position in the plane and its edges are represented by (possibly crossing) straight-line segments between these points. Two edges of $G$ are said to be {\em crossing} if the segments representing them have a point in common.

\begin{conjecture}\label{geomgraph}
 	Let $k\geq 2$ and $n$ be positive integers, and let $G$ be a geometric graph with $n$ vertices, containing no $k$ pairwise crossing edges. Then the number of edges of $G$ is $O_{k}(n)$.
 \end{conjecture}

The result of Capoyleas and Pach mentioned above implies that Conjecture~\ref{geomgraph} holds for geometric graphs $G$, where the points representing the vertices of $G$ form the vertex set of a convex $n$-gon. It is also known to be true for $k\le 4$; see \cite{AgAP}, \cite{Ac}, \cite{AcT}. For $k>4$, it was proved by Valtr~\cite{V} that if a geometric graph on $n$ vertices contains no $k$ pairwise crossing edges then it its number of edges is $O_k(n\log n)$ edges. 
\smallskip

A bipartite variant of Conjecture \ref{mainconj} was proved by Suk \cite{S}. He showed that if $\F\subset 2^{[n]}$ does not contain $2k$ sets $A_{1},\ldots,A_{k}$ and $B_{1},\ldots,B_{k}$ such that $A_{i}$ and $B_{j}$ are crossing for all $i,j\in [k]$, then $|\F|\leq (2k-1)^{2}n$.
\smallskip

The notion of $k$-cross-free families was first introduced by Karzanov~\cite{K} in the context of multicommodity flow problems. Let $G=(V,E)$ be a graph, $X\subset V$. A multiflow $f$ is a fractional packing of paths in $G$. We say that $f$ locks a subset $A\subset X$ in $G$ if the total value of all paths between $A$ and $X\setminus A$ is equal to the minimum number of edges separating $A$ from $X\setminus A$ in $G$.  A family $\F$ of subsets of $X$ is called {\em lockable} if for every graph $G$ with the above property there exists a multiflow $f$ that locks every member $A\in\F$. The celebrated locking theorem of Karzanov and Lomonosov~\cite{KL} states that a set family is lockable if and only if it is {\em $3$-cross-free}. This is a useful extension of the Ford-Fulkerson theorem for network flows, and it generalizes some previous results of Cherkasky~\cite{Ch} and Lov\'asz~\cite{L}; see also~\cite{FKS}.  

 \section{The proof of Theorem \ref{mainthm}}
 In this section, we prove our main theorem. Throughout the proof, floors and ceilings are omitted whenever they are not crucial, and $\log$ stands for the base $2$ logarithm. Also, for convenience,
 we shall use the following extended definition of binomial coefficients: if $x$ is a real number and $k$ is a positive integer,

$$\binom{x}{k}=\begin{cases} \frac{x(x-1)\ldots(x-k+1)}{k!} &\mbox{if } x\geq k-1 \\
0 & \mbox{if } x<k-1. \end{cases}$$
Let us remark that the function $f(x)=\binom{x}{k}$ is monotone increasing and convex.

A pair of sets, $A,B\in 2^{[n]},$ are said to be \emph{weakly crossing}, if $A\setminus B$, $B\setminus A$ and $A\cap B$ are all non-empty. Clearly, if $A$ and $B$ are crossing, then $A$ and $B$ are weakly crossing as well. We call a set family $\F\subset 2^{[n]}$ \emph{weakly $k$-cross-free} if it does not contain $k$ pairwise weakly crossing sets.

 As our first step of the proof, we show that if $\F\subset 2^{[n]}$ is a $k$-cross-free family, then we can pass to a weakly $k$-cross-free family $\F'\subset 2^{[n]}$ by losing a factor of at most $2$ in the cardinality.

 \begin{lemma}\label{weaklycross}
 	Let $\F\subset 2^{[n]}$ be a $k$-cross-free family. Then there exists a weakly $k$-cross-free family $\F'\subset 2^{[n]}$ such that $|\F'|\geq |\F|/2$.
 \end{lemma}

 \begin{proof}
 	Let $$\mathcal{F}'=\{A\in \mathcal{F}':1\not\in A\}\cup \{[n]\setminus A:A\in \mathcal{F}',1\in A\}.$$
 	Clearly, we have $|\F'|\geq |\F|/2$.
 	
 	Note that two sets $A,B\in [n]$ are crossing if and only if $A$ and $[n]\setminus B$ are crossing. Hence, $\mathcal{F}'$ does not contain $k$ pairwise crossing sets. But no set in $\mathcal{F}'$ contains $1$, so we cannot have $A\cup B=[n]$ for any $A,B\in \F'$. Thus, $A,B\in \F'$ are crossing if and only if $A$ and $B$ are weakly crossing. Hence, $\F'$ satisfies the conditions of the lemma.
 \end{proof}

 Now Theorem \ref{mainthm} follows trivially from the combination of Lemma \ref{weaklycross} and the following theorem.

 \begin{theorem}\label{mainthm2}
 Let $k\geq 2$ and $n$ be positive integers and let $\F\subset 2^{[n]}$ be a weakly $k$-cross-free family. Then $|\F|=O_{k}(n\log^{*}n).$	
 \end{theorem}

The rest of this section is devoted to the proof of this theorem. Let us briefly sketch the idea of the proof while introducing some of the main notation.

Let $\F$ be a weakly $k$-cross-free family. First, we shall divide the elements of $\F$ into $\log n$ parts according to their sizes: for $i=0,\ldots,\log n$, let $\mathcal{F}_{i}:=\{X\in \mathcal{F}: 2^{i}< |X|\leq 2^{i+1}\}$. We might refer to the families $\F_{i}$ as \emph{blocks}. Next, we show that, as the block $\F_{i}$ is weakly $k$-cross-free, it must have the following property: a positive proportion of $\F_{i}$ can be covered by a collection of chains $\Gamma_{i}$ with the maximal elements of these chains forming an antichain. These chains are going to be the objects of main interest in our proof.

We show that if $\F$ has too many elements, then we can find $k$ chains $\C_{1}\subset \Gamma_{i_{1}},\ldots,\C_{k}\subset \Gamma_{i_{k}}$ for some $i_{1}<\ldots<i_{k}$, and $k$ elements $C_{j,1}\subset\ldots\subset C_{j,k}$ in each chain $\C_{j}$ such that $C_{j,l}\subset C_{j',l'}$ if $j\leq j'$ and $l\leq l'$, and $C_{j,l}$ and $C_{j',l'}$ are weakly crossing otherwise. But then we arrive to a contradiction since the $k$ sets $C_{1,k},C_{2,k-1},\ldots,C_{k,1}$ are pairwise weakly crossing.

 Now let us show how to execute this argument precisely.

\begin{proof}[Proof of Theorem \ref{mainthm2}.]  Without loss of generality, we can assume that $\F$ does not contain the empty set and $1$-element sets, since by deleting them we decrease the size of $\F$ by at most $n+1$.
	
	Let us remind the reader of the  definition of blocks: for $i=0,1,\ldots,\log n$, we have $$\mathcal{F}_{i}:=\{X\in \mathcal{F}: 2^{i}< |X|\leq 2^{i+1}\}.$$
	The next claim gives an upper bound on the size of an antichain in $\F_{i}$.
	
	\begin{claim}\label{claim1}
		 If $\mathcal{A}\subset \mathcal{F}_{i}$ is an antichain, then $$|\mathcal{A}|\leq \frac{(k-1)n}{2^{i}}.$$
    \end{claim}

    \begin{proof}
    	Suppose that there exists $x\in [n]$ such that $x$ is contained in $k$ sets from $\mathcal{A}$. Then these $k$ sets are pairwise weakly crossing. Hence, every element of $[n]$ is contained in at most $k-1$ of the sets in $\mathcal{A}$, which implies that
    	$$(k-1)n\geq \sum_{A\in \mathcal{A}} |A|\geq |\mathcal{A}|2^{i}.$$
    \end{proof}

    In the next claim, we show that a positive proportion of $\F_{i}$ can be covered by chains whose maximal elements form an antichain. We shall use the following notation concerning chains. If $\C$ is a chain of size $l$, denote its elements by $\C(1)\subset\ldots\subset \C(l)$. Accordingly, let $\min \C=\C(1)$ and $\max \C=\C(l)$.
	
	\begin{claim}\label{claim2}
		For every $i\geq 0$, there exists a collection $\Gamma_{i}$ of chains in $\F_{i}$ such that $\{\max \C:\C\in \Gamma_{i}\}$ is an antichain and
		$$\sum_{\C\in \Gamma_{i}}|\C|\geq \frac{|\mathcal{F}_{i}|}{k-1}.$$
	\end{claim}
	
	\begin{proof}
		
		Let $\M$ be the family of maximal elements  of $\F_{i}$ with respect to containment. For each $M\in \M$, let $\mathcal{H}_{M}\subset \F_{i}$ be a family of sets contained in $M$ such that the system $\{\mathcal{H}_{M}\}_{M\in\M}$  forms a partition of $\F_{i}$.
		
		Note that any two sets in $\mathcal{H}_{M}$ have a nontrivial intersection, as every $A\in \mathcal{H}_{M}$ satisfies $A\subset M$ and $|A|>|M|/2$. Hence, $\mathcal{H}_{M}$ cannot contain an antichain of size $k$, otherwise, these $k$ sets would be pairwise weakly crossing. Therefore, by Dilworth's theorem \cite{D}, $\mathcal{H}_{M}$ contains a chain $\C_{M}$ of size at least $|\mathcal{H}_{M}|/(k-1)$. The collection $\Gamma_{i}=\{\C_{M}:M\in \M\}$ meets the requirements of the Claim.
	\end{proof}

    Let $\Gamma_{i}$ be a collection of chains in $\F_{i}$ satisfying the conditions in Claim~\ref{claim2}.  As the maximal elements of the chains in $\Gamma_{i}$ form an antichain, Claim \ref{claim1} gives the following upper bound on the size of $\Gamma_{i}$:

     \begin{equation}\label{nofchains}
    |\Gamma_{i}|\leq \frac{(k-1)n}{2^{i}}.
    \end{equation}

    From now on, fix some positive real numbers $a,b$ with $ a\leq b\leq \log n$ and consider the union of blocks $\F_{a,b}=\bigcup_{a<i\leq b}\F_{i}$. Analogously, let $\Gamma_{a,b}=\bigcup_{a<i\leq b}\Gamma_{i}$. Allowing $a$ and $b$ to be not necessarily integers will serve as a slight convenience. In what follows, we bound the size of $\F_{a,b}$.

     For each chain $\C\in \Gamma_{a,b}$, define a set $Y(\C)$ by picking an arbitrary element from each of the difference sets ${\C(j+1)\setminus \C(j)}$ for $j=1,\ldots,|\C|-1$, and from $\C_{1}$, as well. Clearly, we have $|Y(\C)|=|\C|$. For every $y\in[n]$, let $d(y)$ be the number of chains $\C$ in $\Gamma_{a,b}$ such that $y\in Y(\C)$. Note that
     \begin{equation}\label{equ2}
    \sum_{y\in [n]}d(y)=\sum_{\C\in\Gamma_{a,b}}|Y(\C)|=\sum_{\C\in\Gamma_{a,b}}|\C|\geq \frac{|\mathcal{F}_{a,b}|}{k-1},
     \end{equation}
     where the last inequality holds by Claim \ref{claim2}.

      We will bound the size of $\F_{a,b}$ by arguing that one cannot have $k$ different elements of $[n]$ appearing in $Y(\C)$ for many different sets $\C\in \Gamma_{a,b}$ without violating the condition that $\F$ is weakly $k$-cross-free. Thus, $\sum_{y\in [n]} d(y)$ must be small. For this, we need the following definition.

     \begin{definition}Let $y\in [n]$. Consider a $k$-tuple of chains $(\C_{1},\ldots,\C_{k})$ in $\Gamma_{a,b}$, where $\C_{i}\in\Gamma_{j_{i}}$ for a strictly increasing sequence $j_{1}<\ldots<j_{k}$. We say that $(C_{1},\ldots,C_{k})$ is \emph{good} for $y$ if
     \begin{description}
     	\item[(i)] $y\in Y(\C_{i})$ for $i\in [k]$,
     	\item[(ii)] if $C_{i}\in\C_{i}$ is the smallest set such that $y\in C_{i}$, then $C_{1}\subset\ldots\subset C_{k}$.    	
     \end{description}
     \end{definition}

     Next, we show that if $d(y)$ is large, then $y$ is good for many $k$-tuples of chains. Let $g(y)$ denote the number of good $k$-tuples for $y$.

     \begin{claim}\label{degree}
     For every $y\in [n]$, we have $$g(y)\geq \binom{d(y)/(k-1)^{2}}{k}.$$
     \end{claim}

     \begin{proof}
     	Let $d=d(y)$ and let $\C_{1},\ldots,\C_{d}\in\Gamma_{a,b}$ be the chains such that $y\in Y(\C_{i})$. Also, for $i=1,\ldots,d$, let $C_{i}$ be the smallest set in $\C_{i}$ containing $y$, and let $\mathcal{H}=\{C_{1},\ldots,C_{d}\}$.
     	
     	The family $\mathcal{H}$ is intersecting. Therefore, it cannot contain an antichain of size $k$, as any two elements of such an antichain are weakly crossing. Applying Dilworth's theorem \cite{D}, we obtain that $\mathcal{H}$ contains a chain of size at least $s=\lceil d/(k-1)\rceil$. Without loss of generality, let $C_{1}\subset\ldots\subset C_{s}$ be such a chain.
     	
     	For any $a\le i\le b$, $\F_{i}$ contains at most $k-1$ members of the sequence $C_{1},\ldots,C_{s}$. Otherwise, if $\C_{j_{1}},\ldots,\C_{j_{k}}\in \Gamma_{i}$ for some $1\leq j_{1}<\ldots<j_{k}\leq s$, the maximal elements $\max\C_{j_{1}},\ldots,\max \C_{j_{k}}$ are pairwise weakly crossing, because these sets form an antichain and contain $y$.
     	
     	This implies that the sets $C_{1},\ldots,C_{s}$ are contained in at least $r=\lceil s/(k-1)\rceil\geq d/(k-1)^2$ different blocks. Thus, we can assume that there exist $i_{1}<\ldots<i_{r}$ and $j_{1}<\ldots<j_{r}$ such that $C_{i_{l}}\in \F_{j_{l}}$ for $l\in [r]$.
     	
     	Then any $k$-element subset of $\{\C_{i_{1}},\ldots,\C_{i_{r}}\}$ is a good $k$-tuple for $y$, resulting in at least
     	$$\binom{r}{k}\geq \binom{d(y)/(k-1)^{2}}{k}$$
     	good $k$-tuples for $y$.
     \end{proof}

    Now we give an upper bound on the total number of $k$-tuples that may be good for some $y\in[n]$. A $k$-tuple of chains in $\Gamma_{a,b}$ is called \emph{nice} if it is good for some $y\in [n]$. Let $N$ be the number of nice $k$-tuples.

     \begin{claim}\label{available}
     	We have
     	$$N<\frac{2(k-1)^{k}n}{2^{a}}\binom{b}{k-1}.$$
     \end{claim}

     \begin{proof}
     	Let $\C\in \Gamma_{a,b}$. Let us count the number of nice $k$-tuples $(\C_{1},\ldots,\C_{k})$ for which $\C=\C_{1}$. Note that in a nice $k$-tuple $(\C_{1},\ldots,\C_{k})$, the set $\min \C_{1}$ is contained in $\max\C_{1},\ldots,\max\C_{k}$.
     	
     	But then, for any positive integer $i$ satisfying $a< i\leq b$, there are at most $k-1$ chains in $\Gamma_{i}$ that can belong to a nice $k$-tuple with first element $\C$. Indeed, suppose that there exist $k$ chains $\mathcal{D}_{1},\ldots,\mathcal{D}_{k}$ in $\Gamma_{i}$ that all appear in a nice $k$-tuple with their first element being $\C$. Then $\{\max\mathcal{D}_{1},\ldots,\max\mathcal{D}_{k}\}$ is an intersecting antichain: it is intersecting because $\max \mathcal{D}_{j}$ contains $\min\C$ for $j\in [k]$, and it is an antichain, by the definition of $\Gamma_{i}$. Thus, any two sets among $\max\mathcal{D}_{1},\ldots,\max\mathcal{D}_{k}$ are weakly crossing, a contradiction.
     	
     	Hence, the number of nice $k$-tuples $(C_{1},\ldots,C_{k})$ for which $\C_{1}=\C$ is at most $\binom{b}{k-1}(k-1)^{k-1}$, as there are at most $\binom{b}{k-1}$ choices for $j_{2}<\ldots<j_{k}\leq b$ such that $C_{l}\in \Gamma_{j_{l}}$ for $l=2,\ldots,k$, and there are at most $k-1$ further choices for each chain $C_{l}$ in $\Gamma_{j_{l}}$.
     	
     	Clearly, the number of choices for $\C=\C_{1}$ is at most the size of $\Gamma_{a,b}$, which is  $$|\Gamma_{a,b}|=\sum_{a<i\leq b}|\Gamma_{i}|\leq \sum_{a<i\leq b}\frac{(k-1)n}{2^{i}}<\frac{2(k-1)n}{2^{a}};$$
     	see (\ref{nofchains}) for the first inequality. Hence, the total number of nice $k$-tuples is at most $$\frac{2(k-1)^{k}n}{2^{a}}\binom{b}{k-1}.$$
     \end{proof}

     The next claim is the key observation in our proof. It tells us that a $k$-tuple of chains cannot be good for $k$ different elements of $[n]$.

     \begin{claim}\label{config}
     	There are no $k$ different elements $y_{1},\ldots,y_{k}\in [n]$ and a $k$-tuple $(\C_{1},\ldots,\C_{k})$ such that $(\C_{1},\ldots,\C_{k})$ is good for $y_{1},\ldots,y_{k}$.
     \end{claim}

     \begin{proof}
     	Suppose that there exist such a $k$-tuple $(\C_{1},\ldots,\C_{k})$ and $k$ elements $y_{1},\ldots,y_{k}$. For $i,j\in [k]$, let $C_{i,j}$ be the smallest set in $\C_{i}$ that contains $y_{j}$. By the definition of a good $k$-tuple, we have $C_{1,j}\subset \ldots\subset C_{k,j}$ for $j\in [k]$. Also, the sets $C_{1,1},\ldots,C_{1,k}$ are distinct elements of the chain $\C_{1}$, so, without loss of generality, we can assume that $C_{1,1}\subset C_{1,2}\subset\ldots\subset C_{1,k}$.
     	
     	First, we show that this assumption forces $C_{i,1}\subset\ldots\subset\C_{i,k}$ for all $i\in [k]$, as well. To this end, it is enough to prove that we cannot have $C_{i,j'}\subset C_{i,j}$ for some $1\leq j<j'\leq k$. Indeed, suppose that $C_{i,j'}\subset C_{i,j}$. Then $y_{j}\in C_{i,j}$, but $y_{j}\not\in C_{i,j'}$. However, $y_{j}\in C_{1,j}$ and $C_{1,j}\subset C_{1,j'}\subset C_{i,j'}$, contradiction.
     	
     	Next, we show that any two sets in the family $$\mathcal{H}:=\big\{C_{i,k+1-i}:i\in [k]\big\}$$ are weakly crossing. Every element of $\mathcal{H}$ contains $C_{1,1}$, so $\mathcal{H}$ is an intersecting family. Our task is reduced to showing that $\mathcal{H}$ is an antichain. Suppose that $C_{i,k+1-i}\subset C_{i',k+1-i'}$ for some $i,i'\in [k]$, $i\neq i'$. Then we must have $i<i'$. Otherwise, $|C_{i,k+1-i}|>|C_{i',k+1-i'}|,$ as $C_{i,k+1-i}\in \F_{j_{i}}$ and $C_{i',k+1-i'}\in \F_{j_{i'}}$ hold for some $j_{i'}<j_{i}$. But if $i<i'$, we have $y_{k+1-i}\in C_{i,k+1-i}$ and $y_{k+1-i}\not\in C_{i',k+1-i'}$, so $C_{i,k+1-i}\not\subset C_{i',k+1-i'}$.
     	
     	Thus, any two sets of the $k$-element family $\mathcal{H}$ are weakly crossing, which is a contradiction.
     \end{proof}

     Let $M$ be the number of pairs $(y,(C_{1},\ldots,C_{k}))$ such that $(C_{1},\ldots,C_{k})$ is a good $k$-tuple for $y\in [n]$. Let us double count $M$.

     On one hand, Claim \ref{config} implies that $M\leq (k-1)N$. Plugging in our upper bound of Claim \ref{available} for $N$, we get
     $$M\leq (k-1)N<\frac{2(k-1)^{k+1}n}{2^{a}}\binom{b}{k-1}\leq \frac{2n(k-1)^{k+1}b^{k-1}}{2^{a}(k-1)!}.$$
     For simplicity, write $c_{1}(k)=2(k-1)^{k+1}/(k-1)!$, then our inequality becomes
     \begin{equation}\label{Mupper}
     M\leq \frac{c_{1}(k)nb^{k-1}}{2^{a}}.
     \end{equation}

     On the other hand, we have
     $$M= \sum_{y\in [n]}g(y),$$
     where $g(y)$, as before, stands for the number of good $k$-tuples for $y$. Applying Claim \ref{degree}, we can bound the right-hand side from below, as follows.
     $$\sum_{y\in [n]}g(y)\geq  \sum_{y\in [n]} \binom{d(y)/(k-1)^{2}}{k}.$$
     Exploiting the convexity of the function $\binom{x}{k}$, Jensen's inequality implies that the right-hand side is at least
     $$n\binom{\sum_{y\in [n]}d(y)/(k-1)^{2}n}{k}.$$
     Finally, using (\ref{equ2}), we obtain
     \begin{equation}\label{Mlower}
     M\geq n\binom{|\F_{a,b}|/(k-1)^{3}n}{k}.
     \end{equation} Suppose that $|\F_{a,b}|>2k(k-1)^{3}n$. In this case, we have $$\binom{|\F_{a,b}|/(k-1)^{3}n}{k}>\left(\frac{|\mathcal{F}_{a,b}|}{2(k-1)^{3}n}\right)^{k}\frac{1}{k!}.$$
     Writing $c_{2}(k)=1/2^{k}(k-1)^{3k}k!$, we can further bound the right-hand side of (\ref{Mlower}) and arrive at the inequality
     \begin{equation}\label{Mlower2}
     M>\frac{c_{2}(k)|\F_{a,b}|^{k}}{n^{k-1}}.
     \end{equation}

     Comparing (\ref{Mupper}) and (\ref{Mlower2}), we obtain
     $$\frac{c_{1}(k)nb^{k-1}}{2^{a}}> \frac{c_{2}(k)|\F_{a,b}|^{k}}{n^{k-1}},$$
     which yields the following upper bound for the size of $\F_{a,b}$:
     $$|\mathcal{F}_{a,b}|<n\left(\frac{c_{1}(k)}{c_{2}(k)}\right)^{1/k}\frac{b^{(k-1)/k}}{2^{a/k}}.$$
     Recall that (\ref{Mlower2}) and the last inequality hold under the assumption that $|\F_{a,b}|>2k(k-1)^{3}n$. Hence, writing $c_{3}(k)=(c_{1}(k)/c_{2}(k))^{1/k}$, we get that
     \begin{equation}\label{finalequ}
     |\mathcal{F}_{a,b}|<\max\left\{2k(k-1)^{3}n,\frac{c_{3}(k)nb^{(k-1)/k}}{2^{a/k}}\right\}
     \end{equation}
     holds without any assumption.

     We finish the proof by choosing an appropriate sequence $\{a_{i}\}_{i=0}^{s}$ and applying the bound (\ref{finalequ}) for the families $\F_{a_{i},a_{i+1}}$.

      Define the sequence $\{a_{i}\}_{i=0,1,\ldots}$ such that $a_{0}=0$, $a_{1}=k^{2}$ and $a_{i+1}=2^{a_{i}/(k-1)}$ for $i=1,2,\ldots$. Let $s$ be the smallest positive integer such that $a_{s}> \log n$. Clearly, we have $s=O_{k}(\log^{*}(n))$.
      Also,
      $$|\mathcal{F}_{a_{0},a_{1}}|=|\mathcal{F}_{0,k^{2}}|\leq \sum_{l=1}^{2^{k^{2}}}\frac{(k-1)n}{l}=O_{k}(n),$$
      as $\F$ has at most $(k-1)n/l$ elements of size $l$ for $l\in [n]$, by the weakly $k$-cross-free property. Finally, for $i=1,\ldots,s-1$, (\ref{finalequ}) yields that
     $$|\mathcal{F}_{a_{i},a_{i+1}}|<
     \max\left\{2k(k-1)^{3}n,\frac{c_{3}(k)na_{i+1}^{(k-1)/k}}{2^{a_{i}/k}}\right\}=
     \max\{2k(k-1)^{3},c_{3}(k)\}n.$$
     The proof of Theorem \ref{mainthm2} can be completed   by noting that	
     $$|\F|= \sum_{i=0}^{s-1}|\mathcal{F}_{a_{i},a_{i+1}}|\leq O_{k}(n)+s\max\{2k(k-1)^{3},c_{3}(k)\}n=O_{k}(n\log^{*} n).$$

\end{proof}

\bigskip

\noindent{\bf Acknowledgement.} We are grateful to Peter Frankl for his useful remarks. In particular, he pointed out that with more careful computation Claim~\ref{degree} can be improved to
$$g(y)\geq (k-1)^k\binom{d(y)/(k-1)^{2}}{k}.$$

\end{document}